\documentclass{jTarXiv}
\usepackage[english]{babel}

\usepackage[utf8]{inputenc}
\usepackage{latexsym}
\usepackage{amssymb}

\usepackage{amsthm}
\usepackage{comment}

\usepackage{xcolor}
\definecolor{darkgrey}{rgb}{0.4,0.4,0.5}
\usepackage{hyperref}
\hypersetup{
    colorlinks=true,
    citecolor=darkgrey
}

\usepackage{algpseudocode}
\usepackage{url}

\usepackage{multirow}

\usepackage{amsmath}

\newtheorem{thm}{Theorem}[section]
\newtheorem{lemma}[thm]{Lemma}
\newtheorem{conjecture}[thm]{Conjecture}

\theoremstyle{definition}

\newtheorem{proposition-definition}[thm]{Proposition-Definition}

\newtheorem{theorem}[thm]{Theorem}

\newtheorem{proposition}[thm]{Proposition}

\newtheorem{problem}[thm]{Problem}

\theoremstyle{remark}

\newtheorem{algorithm}[thm]{Algorithm}
\numberwithin{equation}{section}

\def\Kappa{\mathcal{K}}
\newcommand{\GC}{\mathrm{GC}}
\newcommand{\rank}{\mathrm{rank}}
\newcommand{\GL}{\mathrm{GL}}

\newcommand{\Gal}{\mathrm{Gal}}
\newcommand{\Prob}{\mathrm{Prob}}

\newcommand{\val}{\mathrm{val}}

\newcommand{\lcm}{\mathrm{lcm}}

\newcommand{\Tr}{\mathrm{Tr}}

\newcommand{\id}{\mathrm{id}}
\newcommand{\Norm}{\mathrm{Norm}}

\newcommand{\Q}{\mathbb Q}

\newcommand{\Z}{\mathbb Z}
\newcommand{\N}{\mathbb N}
\newcommand{\F}{\mathbb F}
\newcommand{\R}{\mathbb R}
\newcommand{\OO}{\mathcal O}

\newcommand{\gp}{\mathfrak p}
\newcommand{\gl}{\mathfrak l}
\newcommand{\ga}{\mathfrak a}

\begin{document}

\title[Greenberg's $p$-rationality conjecture]{Numerical verification of the Cohen-Lenstra-Martinet heuristics \\and of Greenberg's $p$-rationality conjecture}

\author{Razvan Barbulescu}
\address{Razvan Barbulescu\\
UMR 5251, CNRS, INP, Université de Bordeaux, \\
351, cours de la Libération, 33400, Talence, France}
\email{razvan.barbulescu@u-bordeaux.fr}

\author{Jishnu Ray}
\address{Jishnu Ray\\
Department of Mathematics, The University of British Columbia\\
Room 121, 1984 Mathematics Road, V6T 1Z2, Vancouver, BC, Canada }
\email{jishnuray1992@gmail.com}


\thanks{}

\maketitle

\begin{resume}
\begin{small}
Dans cet article nous apportons des éléments en faveur de la conjecture de Greenberg d'existence de corps p-rationnels à groupe de Galois connu. Nous intruisons une famille de corps biquadratiques $p$-rationnels et nous donnons des nouveaux exemples numériques de corps $p$-rationnels multiquadratiques de grand degré. Dans le cas des corps multiquadratiques et multicubiques on prouve que la conjecture est une conséquence de la conjonction de l'heuristique de Cohen-Lenstra-Martinet et d'une conjecture de Hofmann et Zhang portant sur le régulateur $p$-adique; nous apportons des nouveauz résultats numériques en faveur de ces conjectures. Une comparaison des outils existants nous amène à proposer des modifiquations algorithmiques.  \end{small}
\end{resume}

\begin{abstract}
In this paper we make a series of  numerical experiments to support  Greenberg's $p$-rationality conjecture, we present a family of $p$-rational biquadratic fields and we find new examples of $p$-rational multiquadratic fields. In the case of multiquadratic and multicubic fields we show that the conjecture is a consequence of the Cohen-Lenstra-Martinet heuristic and of the conjecture of Hofmann and Zhang on the $p$-adic regulator, and we bring new numerical data to support the extensions of these conjectures. We compare the known algorithmic tools and propose some improvements.
\end{abstract}

\section{Introduction}

Let $K$ be a  number field,  $S_p$ the set of prime ideals of $K$ above $p$, $K_{S_p}$ the compositum of all finite $p$-extensions of $K$ which
are unramified outside $S_p$. We call $\mathcal{T}_p$ the torsion subgroup of the abelianization of $\Gal(K_{S_p}/K)$. The study of $\mathcal{T}_p$ is a major question in Iwasawa theory. If $K$ satisfies Leopoldt's conjecture at $p$ and $\mathcal{T}_p\simeq 0$ we say that $K$ is $p$-rational.  A. Movahhedi and T. Nguyen Quang Do, 
 in \cite{Moh1},  discussed this notion of $p$-rational fields and showed that if $K$ is $p$-rational, then  $\Gal(K_{S_p}/K)$ is a free pro-$p$ group (see also the PhD thesis of Movahhedi \cite[Chapter II]{MovThesis}).    Movahhedi also proved an equivalent characterization of $p$-rational fields depending on the class number of $K$ and the unit groups of its $p$-adic completions (cf. \textit{ibid}).  We recommend Gras' book~\cite{Grasbook} for a presentation of numerous results in the topic of $p$-rational fields.

The existence of $p$-rational fields  allows us to obtain algorithms and proofs. For example, Schirokauer~\cite{Sch93} proposed an algorithm to compute discrete logarithms in the field of $p$ elements which uses a $p$-rational number field. A more recent application  due to Greenberg~\cite[Prop 6.7]{Green} is the following. If there exists a totally complex $p$-rational number field $K$ such that $\Gal(K/\Q)\simeq (\Z/2\Z)^t$ for some $t$, then for all integers $n$ such that $4\leq n\leq 2^{t-1}-3$ there exists an explicit continuous representation with open image
\begin{equation*}
	\rho_{n,p} : \Gal(\overline{\Q}/ \Q)\rightarrow \GL(n,\Z_p).
 \end{equation*}  
Note that for $n=2$ one has such a construction using elliptic curves. Yet, the only other results before Greenberg's construction correspond to $n=3$ (due to Hamblen and Upton according to~\cite{Green}). 

Greenberg conjectured that for any pair $(p,t)$ there exists a $p$-rational number field $K$ such that $\Gal(K/\Q)\simeq (\Z/2\Z)^t$. Let us consider the generalization of this conjecture to any abelian field. 
\begin{problem}\label{problem:GC}
	Given a finite abelian group $G$ and a prime $p$, decide if the following statements hold: 
 there exists one (resp. infinitely many) $p$-rational number field(s) of Galois group $G$; in this case we say that Greenberg's conjecture (resp. the infinite version of Greenberg's conjecture) holds for $G$ and $p$ or simply that $\GC(G,p)$ (resp. $\GC_\infty(G,p)$) holds.
\end{problem}

The scope of this article is to investigate this problem. In Section~\ref{sec:examples}, we propose a family of $p$-rational biquadratic fields and prove $\GC((\Z/2\Z)^t,p)$ for all primes $p\in [5,97]$ and $t\in [7,11]$ depending on~$p$. In Section~\ref{sec:density} we prove that $\GC_\infty((\Z/q\Z)^t,p)$ for $q=2$ and $3$ and for any $t\geq 1$ and $p\geq 5$ are consequences of the Cohen-Lenstra-Martinet heuristic and of a recent conjecture of Hofmann and Zhang. Finally, in Section~\ref{sec:algorithmic} we present a comparison and modifications of the algorithms used to obtain the experimental data. 

\section*{Acknowledgments} We are very grateful to Ralph Greenberg  who encouraged us to do this study.  We also thank the referees for very careful reading of our manuscript and for correcting several errors and inaccuracies in the previous versions of this paper. 

\section{Some examples of $p$-rational fields for Greenberg's conjecture}\label{sec:examples}

Let $n_K$, $h_K$, $D_K$ and $E_K$ be the degree, the class number, the discriminant and the unit group of~$K$. If $K$ is abelian, we denote its conductor by~$c_K$.

The first objective of this article is to present an infinite family of $p$-rational fields and to find examples of multiquadratic $p$-rational fields that are larger than the results in~\cite{Green}. For this we use a characterization of $p$-rational fields as follows.

\begin{proposition}[Prop II.1 of \cite{MovThesis}]\label{def:p-rational} We use the notations given above and call $(r_1,r_2)$ the signature of $K$. The following statements are equivalent:
	\begin{enumerate}
		\item $K$ is $p$-rational (i.e. $K$ satisfies Leopoldt's conjecture at $p$ and $\mathcal{T}_p$ is trivial)  or equivalently $\Gal(K_{S_p}/K)^\text{ab}\simeq \Z_p^{1+r_2}$;
        \item $\Gal(K_{S_p}/K)$ is a free pro-$p$ group with  $r_2+1$ generators;
        \item $\Gal(K_{S_p}/K)$ is a free pro-$p$ group.
	\end{enumerate}
	By \cite{Moh2}, the above conditions on $p$-rationality  are also equivalent to
	\begin{enumerate}
		\setcounter{enumi}{3}
		\item 
		\begin{enumerate}
			\item $\left\{\alpha \in K^{\times} \mid
			\begin{array}{l}
			\alpha
			\mathcal{O}_K=\mathfrak{a}^p\text{ for some fractional ideal }\ga\\
			\text{ and
			}\alpha \in
			(K_{\gp}^{\times})^p\text{  for all } \gp \in S_p
			\end{array}
			\right\}=(K^{\times})^p $,
			\item  and
			the map $\mu(K)_p \rightarrow \prod_{\gp \in
				S_p}\mu(K_{\gp})_p$ is an isomorphism,
		\end{enumerate}
        where $K_\mathfrak{p}$ is the completion of $K$ at a prime ideal $\mathfrak{p}\in S_p$ and $\mu(K)_p$ is the set of $p$-th roots of unity in $K$.
	\end{enumerate}
\end{proposition}
We  note that the condition (4.b) is automatically satisfied for primes  $p>+1$ as $[\Q_p(\zeta_p):\Q_p]=p-1$ and therefore no $p$-adic completion of $K$ can contain $p$-th roots of unity.   

We call $p$-primary any unit of $K$ which is not a $p$-th power in $K$ but it's a $p$-th power in all the $p$-adic completions of $K$. Assume that $K$ satisfies Leopoldt's conjecture (e.g. $\Gal(K/\Q)$ is abelian) and that $p$ is such that the map $\mu(K)_p \rightarrow \prod_{\gp \in S_p}\mu(K_{\gp})_p$ is an isomorphism (e.g.  $p>n_K+1$). Then we have a simple criterion for $p$-rationality: if $p\nmid h_K$ and $K$ has no $p$-primary units then $K$ is $p$-rational. 
	
In particular, for all primes $p\geq 5$, all imaginary quadratic fields $K$ such that $p\nmid h_K$ are $p$-rational. Hence $\GC_\infty(\Z/2\Z,p)$ is a consequence of a result due to Hartung.

\begin{proposition}[\cite{Hartung}]
For all odd primes $p$ there exist infinitely many square-free integers $D<0$
such that $h_{\Q(\sqrt{D})}\cdot D\not\equiv 0\pmod p$. Therefore, there exist infinitely many $p$-rational imaginary quadratic fields. As a consequence,  $\GC_\infty(\Z/2\Z,p)$ holds. 
\end{proposition}

The existence of $p$-primary units is easily tested using $p$-adic logarithms (\cite{Washington} Sec. 5.1). Let $K$ be a number field (not necessarily abelian) and $p$ a prime which is unramified in~$K$ and such that $K$ has no $p$-th roots of unity. In the following $\OO_p=\mathbb{Z}_p \otimes \mathcal{O}_K$  and we set $e_p:=\lcm(\{\Norm(\gp)-1\text{ : }\gp\in S_p\})$ and $K_p:=\{x\in K^*\mid \forall \gp\in S_p, \val_\gp(x)=0\}$. 
Since $x\mapsto x^{e_p}$ injects $K_p$ into $\{z\in \mathbb{C}_p^* : \forall \gp\in S_p, \val_\gp(z-1) \geq 1 \}$, we can extend $\log_p$ to $K_p$ by $\log_p(x):=\frac{1}{e_p}\log_p (x^{e_p})$. 

 Note that an element of $\OO_K$ is a $p$-th power in $K_\gp$ for all $\gp\in S_p$ if and only if $\log_p(x)\in p^2\OO_p$. Hence, a unit $\varepsilon\in K\backslash K^p$ is $p$-primary if and only if $\log_p(\varepsilon)\equiv 0\pmod{p^2\OO_p}$. 

Assume that $K$ is totally real. If $U$ is a set of $n_K-1$ units, we denote by $R_p(U)$ the $p$-adic regulator (\cite{Washington} Sec 5.5 and \cite{HofZha16} Sec. 2.1); if $U$ is a system of fundamental units we simply write $R_{K,p}$. We call normalized $p$-adic regulator the quotient $R'_{K,p}:=R_{K,p}/p^{n_K-1}$ and note that if $K$ has $p$-primary units then $R'_{K,p}\in p\OO_p$.

If on the contrary, $p$ is ramified at $K$ we don't have necessarily that $R_p'\in \OO_p$ and we don't have an equivalence for existence of $p$-primary units. However, when $p$-primary units are present it remains true that $\val_p(R'_p)\geq 1$. In the sequel we write ``$p\mid R_p'$'' for ``$\val_p(R_p')\geq 1$'' in the both cases when $p$ is ramified and unramified in~$K$.

In the case of multiquadratic fields we shall need a result of Greenberg:
\begin{lemma}(\cite[Prop 3.6]{Green}) \label{lem:prational_subfield}
Let $K$ be an abelian extension of $\Q$ and $p$ is a prime which does not divide the degree $[K:\Q]$. Then $K$ is $p$-rational if and only if all its  cyclic subfields are $p$-rational.
\end{lemma}

 A study of the $p$-adic logarithm of the fundamental unit allows to construct a $p$-rational biquadratic number field for a fixed prime $p$, i.e. to show that $\GC(\Z/2\Z\times \Z/2\Z,p)$ holds for any~$p$. 
\begin{theorem}{
		For any prime $p$, the field $K=\Q(i\sqrt{p-1},i\sqrt{p+1})$ is
		$p$-rational.} 
\end{theorem}
\begin{proof}
	Let us call $k_1=\Q(\sqrt{p^2-1})$, $k_2=\Q(i\sqrt{p-1})$ and $k_3=\Q(i\sqrt{p+1})$ the three quadratic subfields of~$K$. We treat first the case where $p\geq 5$ using the $p$-rationality criterion presented above: we show that $p\nmid h_{k_1}$ (step 1) and that the fundamental unit of $k_1$ is not $p$-primary (step 2), so $\Q(\sqrt{p^2-1})$ is $p$-rational. Then we show that $\max(h_{k_2},h_{k_3})<p$ (step 3), which shows that $k_2$ and $k_3$ are $p$-rational. This completes the proof for $p\geq 5$ using Lemma~\ref{lem:prational_subfield}. The cases $p=2$ and $p=3$ are treated at the end (step 4). 
		
		\textit{Proof of step 1.	}
		We distinguish two cases depending whether $p$ is of the form $\frac{1}{2}a^2\pm 1$ for some $a\in\Z$.  

\underline{The case when $p\neq\frac{1}{2}a^2\pm 1$ for any $a\in\Z$.} 

Let us show that $\varepsilon=p+\sqrt{p^2-1}$ is a
		fundamental unit. 
		Note first that $D_{k_1}=4Q$ or $Q$ where $Q$ is the square free part of $(p^2-1)/4$, so $D_{k_1}$ is a positive divisor of $p^2-1$.
		Also note that the minimal polynomial of $\varepsilon$ is $\mu_\varepsilon=x^2-2px+1$. If $\varepsilon$ is a square in $k_1$ then $x^4-2px^2+1$ is divisible in $\Q[x]$ by a polynomial of the form $\mu_{\sqrt{\varepsilon}}=x^2-2ax\pm 1$ with $a\in\Z$, which is forbidden by the assumption that $p$ is not of the form $\frac{1}{2}a^2\pm1$. 
        As a real field, $k_1$ has no roots of unity other than $\pm1$ so there exists an odd integer $n$ such that $\varepsilon=\varepsilon_0^n$ where $\varepsilon_0$ is the fundamental unit greater than $1$. Note that  $\gamma:=-(\varepsilon_0^n+\varepsilon_0^{-n})/(\varepsilon_0+\varepsilon_0^{-1})$ belongs to $\Z[\varepsilon_0]$ and therefore to~$\OO_{k_1}$. Since  $\gamma=\Tr(\varepsilon)/\Tr(\varepsilon_0)$ ($\Tr$ denotes the trace), it belongs to $\Q$ and therefore $\gamma$ is an integer, so $\Tr(\varepsilon_0)\in\{\pm 2p,\pm p,\pm1, {\pm2} \}$. Hence the minimal polynomial of $\varepsilon_0$ is equal to $x^2\pm 2p x\pm 1$, $x^2\pm px\pm 1$, $x^2\pm 2x\pm 1$ or $x^2 \pm x \pm 1$. We rule out $x^2\pm2x+1$ because they are not irreducible and we rule out the cases $x^2\pm 2x-1$ and $x^2 \pm x \pm 1$ because their roots belong to $\Q(\sqrt{2})$, $\Q(\sqrt{3})$ or $\Q(\sqrt{5})$, which can only belong to $k_1$ if $p^2-1=\frac{1}{2}a^2$, $p^2-1=3a^2$ or $p^2-1=5a^2$. The first  case is forbidden by our assumption on $p$ and the other two are forbidden because the systems of equations $\{$ $p\pm1=3b^2$ and $p\mp1= c^2$ $\}$ and $\{$ $p\pm1=5b^2$ and $p\mp 1=c^2$ $\}$ have no solutions modulo $4$ for odd $p$ and hence no solutions in $\Z$. Among the remaining polynomials, the only irreducible polynomial whose discriminant divides $p^2-1$ (and satisfying $\varepsilon=\varepsilon_0^n$) is $\mu_\varepsilon$, so $\varepsilon_0=\varepsilon$ (i.e. $n=1$) or equivalently $p+\sqrt{p^2-1}$ is a fundamental unit. 
	
	By a result of Louboutin~\cite[Theorem 1]{Lou1998} we have the following effective bound 
	$$h_{k_1}\leq \sqrt{D_{k_1}}\frac{e\log (D_{k_1})}{4\log\varepsilon}.$$ Since $D_{k_1}\leq p^2-1$, we conclude that $h_{k_1}<p$ and hence $p\nmid h_{k_1}$.

     \underline{The case when $p=\frac{1}{2}a^2\pm 1$ with $a\in\Z$.} 
     
     Let $d$ be the square free part of $p^2-1$ and $\varepsilon:=a+b\omega$ with $a,b\in\Z$ be a fundamental unit of $\Q(\sqrt{p^2-1})$, where $\omega=\sqrt{d}$ or $\frac{1+\sqrt{d}}{2}$ depending on the residue of $d$ mod~$4$. Without loss of generality we can assume that $a>0$ and $|\varepsilon|>1$. Since the conjugate of $\varepsilon$, $a-b\omega $ is also a fundamental unit we have $|a+b\omega|>|a-b\omega|$, so $b\geq 1$. Hence we have $$\varepsilon\geq1+1\cdot \omega\geq  1+\min(\sqrt{d},\frac{1+\sqrt{d}}{2})\geq 1+\sqrt{3}.$$

Note as before that $D_{k_1}=4Q$ or $Q$ where $Q$ is the free part of $p^2-1$. Since $p=\frac{1}{2}a^2\pm 1$, $Q$ is a divisor of $(p\pm 1)/2$, so $D_{k_1}\leq 2(p+1)$. 
We apply Louboutin's bound once again and obtain
$$h_{k_1}\leq \sqrt{2(p+1)}\frac{e\log(\sqrt{2(p+1)})}{4\log \varepsilon} <p,$$
because $p\geq 7$, so $p\nmid h_{k_1}$.  
	
	\textit{Proof of step 2.}
 To test if $\varepsilon$ is $p$-primary we test if $\varepsilon^{p^2-1}-1\equiv 0\pmod{p^2\OO_p}$. Indeed, $\log_p(\varepsilon)=\frac{1}{p^2-1}\log_p(\varepsilon^{p^2-1})\equiv 1 -\varepsilon^{p^2-1}\pmod {p^2\OO_p}$.  Then modulo   $p^2 \Z[\sqrt{p^2-1}$ we have

 $$
	\begin{array}{ccll}
\varepsilon^{p^2-1}-1&\equiv&(p^2-1)^\frac{p^2-1}{2}-1+p(p^2-1)^\frac{p^2-3}{2}\sqrt{p^2-1}&({p^2\Z[\sqrt{p^2-1}]})\\
		&\equiv&\pm p\sqrt{p^2-1}& ( p^2 \Z[\sqrt{p^2-1}]).
\end{array}
$$
	Since $p^2\Z[\sqrt{p^2-1}]\subset p^2\OO_{k_1}\subset p^2\OO_p$, this shows
		that the $p$-adic logarithm of $\varepsilon$ is not a multiple of $p^2$, so
		$\varepsilon$ is not $p$-primary.

	\textit{Proof of step 3.} Recall Hua's bound  for the class numbers of imaginary quadratic fields $k$ (Remark 4. in~\cite{Louboutin-imaginary}, where a sharper inegality is proven in Theorem 3.):
$$h_k\leq \frac{\sqrt{D_k}}{\omega_k}(\log(\sqrt{D_k})+1),$$
   where $\omega_k$ is the number of roots of unity in~$k$; note that $\omega_k\geq 2$. Since for $i=2,3$, $D_{k_i}\leq 4(p+1)$ and for all $x\geq 5$ it holds $\frac{1}{2}\sqrt{4(x+1)}(\log(\sqrt{4x+1})+1)< x$, we conclude that $\max(h_{k_2},h_{k_3})<p$. 

Finally, note that $k_2$ and $k_3$ are imaginary, so they have no $p$-primary units. In the case where $p\geq 5$, the $p$-rationality criterion above applies and we conclude that $k_2$ and $k_3$ are $p$-rational. Hence, all the quadratic subfields of $K$ are $p$-rational and so does~$K$.

         \textit{Proof of step 4.} Let us consider the case of $p=2$ and $p=3$. Note that $S_2$ associated to $\Q(\sqrt{3})$ corresponds to a singleton $K_\gp$ and $\mu(\Q(\sqrt{3}))_2=\{\pm 1\}=\mu(\Q(\sqrt{3})_2)_2$. Also note that $S_3$ associated to $\Q(\sqrt{2})$ is an inert ideal and for the corresponding completion $K_p$ has no $3$-rd roots of unity. In both cases, the condition (4.b) in~Proposition~\ref{def:p-rational} is satisfied and one can apply the $p$-rationality criterion. The unit $\varepsilon:=p-\sqrt{p^2-1}$ is not a $p$-th power locally, by the same argument as above. Since $\varepsilon$ is a power of the fundamental unit, $\Q(\sqrt{p^2-1})$ has no $p$-primary units. Since the class number of $\Q(\sqrt{2})$ and $\Q(\sqrt{3})$ is $1$ we obtain that $k_1$ is $p$-rational. 
         
    As the class numbers of $\Q(\sqrt{-1})$, $\Q(\sqrt{-2})$ and $\Q(\sqrt{-3})$ is $1$, the Example (c) of~\cite[Ch. II]{MovThesis} applies : assuming the class number is not divisible by $p$, the imaginary field $\Q(\sqrt{-d})$ is $2$-rational if and only if $d\not \equiv 7\pmod 8$ and it is $3$-rational if and only if $d=3$ or $d\not \equiv 3\pmod 9$. Indeed, for $p=2$ the squarefree parts of $p-1$ and $p+1$ are $1$ and $3$ which are not congrent to $7$~mod~$8$. For $p=3$ the squarefree parts of $p-1$ and $p+1$ are $2$ and $1$ which are not congruent to $3$ modulo~$9$. Hence, $k_2$ and $k_3$ are $p$-rational and we conclude that $K$ is $p$-rational. 

\end{proof}

\subsection{Some numerical examples of $p$-rational multiquadratic fields}\label{ex:quadratic}

In Table~\ref{tab:composita} we  give examples of complex $p$-rational fields $K$ of Galois group $G=(\Z/2\Z)^t$ for all primes $p\in[5,97]$ and greater values of $t$ than those found by Greenberg and Pollack~\cite[Sec 4.2]{Green}. We emphasize the fact that every example proves the existence of open continuous representations of $\Gal(\overline{\Q}/\Q)$ in $\GL(n,\Z_p)$ by including the values of $n$ corresponding to each field (cf Prop 6.7 in~\cite{Green}). 

The fields $\Q(\sqrt{d_1},\ldots,\sqrt{d_t})$ in the examples were found by searching minimal values of $d_i$ for each value of $i$: we took $d_1$ equal to the smallest positive non-square non divisible by $p$ such that $p\nmid h_{\Q(\sqrt{d_1})}$ and its  fundamental unit is not $p$-primary. For $i=2,3\ldots,t-1$ we computed the smallest $d_i \geq d_{i-1}+1$ relatively prime to $p\prod_{j=1}^{i-1}d_j$ such that all the $2^i-1$ quadratic subfields of $\Q(\sqrt{d_1},\ldots,\sqrt{d_i})$ have class numbers non divisible by $p$ and fundamental units which are not $p$-primary. Finally, $d_t<0$ is the negative integer of smallest absolute value such that $\gcd(d_t,p\prod_{i=1}^{t-1}d_i)=1$ and the $2^{t-1}$ imaginary quadratic subfields of $\Q(\sqrt{d_1},\ldots,\sqrt{d_t})$ have class numbers non divisible by~$p$ (the corresponding scripts are available in the online complement~\cite[search-example.sage]{OnlineComplement}). 

\begin{table}
	\begin{tabular}{llll}
		$p$      & $t$ & $d_1,\ldots,d_t$& open image of $\Gal(\overline{\Q}/\Q)$ in\\
		\hline
		\hline
		5  & 7 &2,3,11,47,97,4691,-178290313 & $\forall n\in[4,61]$, $\GL(n,\Z_5)$\\
		7  & 7 &2,5,11,17,41,619,-816371, & $\forall n\in[4,61]$, $\GL(n,\Z_7)$ \\
		11 &8 &2,3,5,7,37,101,5501,-1193167 & $\forall n\in[4,125]$, $\GL(n,\Z_{11})$ \\
		13 &8 &3,5,7,11,19,73,1097,-85279 & $\forall n\in[4,125]$, $\GL(n,\Z_{13})$\\
		17 &8 &2,3,5,11,13,37,277,-203 & $\forall n\in[4,125]$, $\GL(n,\Z_{17})$ \\
		19 &9 &2,3,5,7,29,31,59,12461, -7663849 & $\forall n\in[4,253]$, $\GL(n,\Z_{19})$  \\
		23 &9 &2,3,5,11,13,19,59,2803,-194377 & $\forall n\in[4,253]$, $\GL(n,\Z_{23})$ \\
		29 &9 &2,3,5,7,13,17,59,293,-11 & $\forall n\in[4,253]$, $\GL(n,\Z_{29})$   \\
		31 &9 &3,5,7,11,13,17,53,326,-8137 & $\forall n\in[4,253]$, $\GL(n,\Z_{31})$ \\
		37 &9 &2,3,5,19,23,31,43,569,-523 & $\forall n\in[4,253]$, $\GL(n,\Z_{37})$ \\
		41 &9 &2,3,5,11,13,17,19,241,-1 & $\forall n\in[4,253]$, $\GL(n,\Z_{41})$ \\
		43 &10 &2,3,5,13,17,29,31,127,511,-2465249 & $\forall n\in[4,509]$, $\GL(n,\Z_{43})$   \\
		47 &10&2,3,5,7,11,13,17,113,349,-1777 & $\forall n\in[4,509]$, $\GL(n,\Z_{47})$ \\
		53 &10&2,3,5,7,11,13,17,73,181,-1213 & $\forall n\in[4,509]$, $\GL(n,\Z_{53})$    \\
		59 &10&2,3,5,11,13,17,31,257,1392,-185401 & $\forall n\in[4,509]$, $\GL(n,\Z_{59})$   \\
		61 &10&2,3,5,7,13,17,29,83,137, -24383 & $\forall n\in[4,509]$, $\GL(n,\Z_{61})$  \\
		67 &11&2,3,5,7,11,13,17,31,47,5011,-2131 & $\forall n\in[4,1023]$, $\GL(n,\Z_{67})$\\
		71 &10&2,3,5,11,13,17,19,59, 79,-943 & $\forall n\in[4,509]$, $\GL(n,\Z_{71})$\\
		73 &10&2,3,5,7,13,17,23,37,61,-1& $\forall n\in[4,509]$, $\GL(n,\Z_{73})$ \\
		79 &10&2,3,5,7,11,23,29,103,107,-1 & $\forall n\in[4,509]$, $\GL(n,\Z_{79})$   \\
		83 &10&2,3,5,7,11,13,17,43,97,-1 & $\forall n\in[4,509]$, $\GL(n,\Z_{83})$  \\
		89 &11&2,3,5,7,11,23,31,41,97,401,-425791 & $\forall n\in[4,1023]$, $\GL(n,\Z_{89})$ \\
		97 &11&2,3,5,7,11,13,19,23,43,73,-1 & $\forall n\in[4,1023]$, $\GL(n,\Z_{97})$  \\
		\hline
	\end{tabular}
	\caption{Examples of $p$-rational complex number fields of the form $K=\Q(\sqrt{d_1},\ldots,\sqrt{d_t})$ such that $\Gal(K)\simeq (\Z/2\Z)^t$ and their consequences on the existence of continuous representations of $\Gal(\overline{\Q}/\Q)$ with open image.}
	\label{tab:composita}
\end{table}

Note that the difference $d_i-d_{i-1}$ increases rapidly so that the cost of finding $p$-rational fields with larger $t$ increases in accordance.  This raises the question of the existence of a natural density of $p$-rational number fields having a given Galois group. 

\section{The Cohen-Lenstra-Martinet heuristic and the conjectured density of $p$-rational fields}\label{sec:density}

Cohen and Lenstra~\cite{Cohen1} and Cohen and Martinet~\cite{Martinet2} conjectured that there exists a natural density of number fields whose class number is divisible by a prime $p$ among the set of number fields of given Galois group and signature. We bring new numeric data in favor of the conjecture in Section~\ref{ssec:evidence CLM}. We recall and extend a conjecture of Hofmann and Zhang about the valuation of the $p$-adic regulator (Section~\ref{ssec:evidence regulator}) and then we prove that these conjectures imply Greenberg's $p$-rationality conjecture (Section~\ref{ssec:implications}).

\subsection{New numerical data to verify the Cohen-Lenstra-Martinet heuristics}\label{ssec:evidence CLM}
The Cohen-Lenstra-Martinet conjecture on cyclic cubic fields~\cite[Conjecture C14]{Cohen1}\cite[Sec. 2, Ex 2(b)]{Martinet2} was initially supported by the data computed on the $2536$ cyclic cubic fields of conductor less than $16000$, i.e., discriminant less than $2.56\cdot 10^6$, (cf.~\cite{Martinet2}). Malle~\cite{Malle} noted that the aforementioned data fit equally well the value and the double of the value predicted by the Cohen-Lenstra-Martinet conjectures. This ambiguity is solved if the computations are pushed up to larger conductors. 

We used PARI/GP (\textbf{http://pari.math.u-bordeaux.fr}) to test the Cohen-Lenstra-Martinet heuristic on the 1585249 cyclic cubic fields of conductor less than $10^7$, e.g., discriminant less than $10^{14}$. The results are summarized in Table~\ref{table:numeric CLM} and the complete data are available in the online complement~\cite[table4.txt.gz]{OnlineComplement}. The data in Table~\ref{table:numeric CLM} show that the relative error between the computed density and the one predicted by Conjecture~\ref{conj:reformulation Martinet} is between $0.2\%$ and $78.3\%$.

\renewcommand{\arraystretch}{1.8}
\begin{table}[h!]
	\begin{tabular}{|c|c|c|c|c|c|}
	    \hline
	    \multirow{2}{*}{$p$}          &
	    Conj. C-14 \cite{Cohen1}
			 & 	
			\multicolumn{2}{c|}{			$\Prob(p\mid h_k :  D_K\leq X^2)$} &
			\multicolumn{2}{c|}{           $\frac{\Prob(p\mid h_K, D_K\leq X^2)-\Prob(p\mid h_K) }{\Prob(p\mid h_K)}$   }\\
\cline{3-6}
	&  $\Prob(p \mid h_K)$  & $X=1.6\cdot 10^3$ & $X=10^7$ & $X=1.6\cdot 10^3$ & $X=10^7$\\
		\hline 
		$5$           &  $1.67\cdot 10^{-3}$    &$\frac{4}{2536}$ &  
		$\frac{3042}{1585249}\approx 1.91\cdot 10^{-3} $ &$-5.4\%$  & $15.1\%$   \\
		\hline
		$7$           &  $4.69\cdot 10^{-2}$     & $\frac{87}{2536}$   &  $\frac{72142}{1585249}\approx  4.55\cdot 10^{-2}$
		&$-26.9\%$& $3.0\%$  \\
		\hline
		$11$          &  $6.89\cdot 10^{-5}$  & $0$ & 
		$\frac{127}{1585249}\approx 8.01\cdot 10^{-5} $& $100\%$  & $16.3\%$   \\
		\hline
		$13$          &  $1.28\cdot 10^{-2}$    & $\frac{26}{2536}$ &  $\frac{20244}{1585249}\approx1.28\cdot 10^{-2}
		$ &$-29.7\%$ & $0.2\%$ \\
		\hline
		$17$          &  $1.20\cdot 10^{-5}$     & $0$ &  $\frac{23}{1585249}\approx1.45\cdot 10^{-5}
		$ &$100\%$ & $20.8\%$ \\
\hline
$19$          &  $5.84\cdot 10^{-3}$     & $\frac{16}{2536}$&   $\frac{9406}{1585249}\approx 9.41\cdot 10^{-3}$ & $8.1\%$ &$1.6\%$ \\
		\hline
		$23$          &  $3.58\cdot 10^{-6}$     & $0$ &
		$\frac{9}{1585249}\approx5.67\cdot 10^{-6}
		$ &  $100\%$ & $58.6\%$ \\
		\hline
		$29$          &  $1.41\cdot 10^{-6}  $     & $0$ &  $\frac{4}{1585249}\approx
		2.52\cdot 10^{-6}$ &$100\%$ & $78.3\%$ \\
		\hline
	\end{tabular}
	\caption{Comparison, for primes $p$ between $5$ and $29$, between the proportion of cyclic cubic fields of conductor less than $X=16000$ (resp. $10^7$) whose class number is divisible by $p$, denoted by $\Prob(p\mid h_K: D_K\leq X^2)$, and the density in~Conjecture C-14 in~\cite{Cohen1}. 
	}
	\label{table:numeric CLM}
\end{table}

If $\Gal(K)\simeq (\Z/q\Z)^t$ for some prime $q\neq p$, then Kuroda's formula~\cite[ Eq (17)]{Kuroda} states that $
	h_K=q^\alpha\prod_{k_i\text{ subfield of degree }q} h_{k_i}$ for some $\alpha\in \N$. In the Cohen-Lenstra-Martinet philosophy, the class numbers of the subfields in Kuroda's formula behave ``independently'', e.g. compare the values predicted for the Galois group $\Z/2\Z\times \Z/2\Z$ to the cube of that of $\Z/2\Z$ as well as the value for $\Z/3\Z\times\Z/3\Z$ to the $4$-th power of that for $\Z/3\Z$. This allows us to extend  Conjectures C-7 and C-14 in \cite{Cohen1} as follows. 

\begin{conjecture}\label{conj:reformulation Martinet} Set $(p)_{\infty}:=\prod_{k \geq 1}(1-p^{-k})$ and $(p)_1:=(1-p^{-1})$.  ~\\
	(1) If $K$ is a real field such that $\Gal(K/\Q)\simeq (\Z/2\Z)^t$  for some $t$ and $p$ is an odd prime, then  
		$$\Prob(p \nmid h_K)=\frac{(p)_\infty}{(p)_1}^{2^t-1}.$$

\noindent(2) If $\Gal(K/\Q)\simeq (\Z/3\Z)^t$ for some $t$ and $p\geq 5$ is a prime then
		$$\Prob(p \nmid h_K)=\left\{ 
		\begin{array}{ll}
		(\frac{(p)_\infty^2}{(p)_1^2})^{\frac{3^t-1}{2}}, &\text{if }p\equiv 1\pmod 3;\\ 
		\frac{(p^2)_\infty}{(p^2)_1}^{\frac{3^t-1}{2}}, &\text{if }p\equiv 2\pmod 3.
		\end{array}
		\right.$$
\end{conjecture}
In the case of Galois group $\Z/3\Z\times \Z/3\Z$ we support Conjecture~\ref{conj:reformulation Martinet} with numerical data which are summarized in Table~\ref{tab:numeric Z/3xZ/3} are available online at~\cite[table5.txt.gz]{OnlineComplement}.

\renewcommand{\arraystretch}{1.6}
\begin{table}[h!]
	\begin{tabular}{|c|c|cc|}
		\hline
		$p$          &\begin{tabular}{c}theoretic\\[-10pt]
			density\end{tabular} & 
		\begin{tabular}{c}
			stat. density\\[-10pt]
			conductor $\leq 10^6$
		\end{tabular}
		&
		\begin{tabular}{c}
			\text{relative}\\[-10pt]
			\text{error}
		\end{tabular}
		\\
		\hline 
		$5$           &  $0.00334$    & $\frac{933}{203559}\approx 0.0066458$ & $31\%$ \\
		\hline
		$7$           &  $0.17481$     & $\frac{23912}{203559}\approx 0.11746$ & $33\%$ \\
		\hline
		$11$          &  $0.00028$   & $\frac{26}{203559}\approx 0.00013$ & $54\%$ \\
		\hline
		$13$          &  $0.02316$     & $\frac{6432}{203559}\approx 0.03160$ & $36\%$ \\
		\hline
		$17$          &  $0.000048$  & $\frac{4}{203559}\approx0.0000197$ & $59\%$ \\
		\hline
		$19$          &  $0.02315$     & $\frac{3536}{203559}\approx0.01737$ &$25\%$ \\
		\hline
	\end{tabular}
	\caption{Statistics on the density of fields of Galois group
		$\Z/3\Z\times\Z/3\Z$ whose class
		number is divisible by $p$ for primes $p$ between $5$ and $19$.  
	}
	\label{tab:numeric Z/3xZ/3}
\end{table}

\subsection{Numerical verification of a conjecture on the $p$-adic regulator}\label{ssec:evidence regulator}

 In a heuristic, Schirokauer~\cite[p. 415]{Sch93} obtained that the density of number fields which contain $p$-primary units is $O(\frac{1}{p})$. The same heuristic implies that the density of fields such that $p$ divides $R'_{K,p}$ is also~$O(\frac{1}{p})$.  Hofmann and Zhang~\cite[Conj 1.1]{HofZha16} go beyond the $O(\frac{1}{p})$ upper bound and make a conjecture on the precise density of cyclic cubic fields such that $p\mid R'_{K,p}$: the density is $\frac{2}{p}-\frac{1}{p^2}$ if $p\equiv 1\pmod 3$ and $\frac{1}{p^2}$ if $p\equiv 2\pmod 3$. In the same philosophy, the normalized $p$-adic regulator of a real quadratic field $K$ is heuristically considered to be random element of $\Z_p$ and therefore the probability that $p$ divides $R'_{K,p}$ is $1/p$.   

If $K$ is a number field such that $\Gal(K)\simeq (\Z/q\Z)^t$ for a prime $q$ and an integer $t$ and if $p\neq q$ is a prime, then Kuroda~\cite[Eq. (18)]{Kuroda} showed that  $R_K=q^\beta\prod_{k_i\text{ subfield of degree }q} R_{k_i}$ for some $\beta\in \N$ where $R_K$ and respectively $R_{k_i}$ denote the regulator of $K$ and of  the fields~$k_i$, respectively. Their proof translates in a verbatim manner to the $p$-adic regulators and the normalized $p$-adic regulators.

We make the heuristic that the $p$-adic regulators of the cyclic subfields of a field of Galois group $(\Z/q\Z)^t$ behave ``independently'' and extend Conjecture 1.1 in~\cite{HofZha16}.
\begin{conjecture}\label{conj:independent reg}
	Let $q=2$ or $3$, $p>q+1$ a prime and $t$ an integer. The probabilities below are among the field $K$ where $p$ is unramified. Then the density of totally real number fields $K$ such that $\Gal(K)=(\Z/q\Z)^t$ for which the normalized $p$-adic regulator is divisible by $p$ is

\noindent (1) $\Prob\left(p\text{ divides }R'_{K,p} : K \text{ real, }\Gal(K)\simeq (\Z/2\Z)^t\right)=1-(1-\frac{1}{p})^{2^t-1}$.

\noindent (2) $\Prob\left(p\text{ divides }R'_{K,p} : \Gal(K)\simeq (\Z/3\Z)^t\right)=1-(1-\mathcal{P})^{\frac{3^t-1}{2}}$, 
		where $$\mathcal{P}=\left\{
		\begin{array}{ll}
		\frac{2}{p}-\frac{1}{p^2}, & \text{ if }p\equiv 1\pmod 3\\
		\frac{1}{p^2}, &\text{ if }p\equiv 1\pmod 3
		\end{array}
		\right.$$ 
If $q=3$, $p\equiv 1\pmod 3$, $t=1$ and the probability concerns the set of fields $K$ which are ramified at $p$, then $\mathcal{P}=\frac{2}{p}$. Overall, if $q=2$ or $3$, with no condition on $K$, $\mathcal{P}\leq \frac{2}{p}$.
\end{conjecture}

We  numerically verified the Conjecture \ref{conj:independent reg} as summarized in Table~\ref{tab:reg 2-3}; the programme can be downloaded from the online  complement~\cite[table6.txt.gz]{OnlineComplement}.

\begin{table}[h!]
	\begin{tabular}{|c|c|c|c|c|}
		\hline
\multirow{2}{*}{$\Gal(K)$} &		\multirow{2}{*}{$p$}   &  experimental & Conj~\ref{conj:independent reg} & relative \\[-8pt] 
\multirow{5}{*}{$\Z/2\Z$}   &		& density  & density & error\\

		\hline
	&	$5$   & $\frac{120037}{607925}\approx 0.20$  &  $0.20$   & $1\%$   \\
		\cline{2-5}
	&	$7$   & $\frac{86702}{607925}\approx 0.14$    & $0.14$    &  $<1\%$ \\
		\cline{2-5}
	&	$11$   & $\frac{54626}{607925}\approx 0.09$    & $0.09$    &  $<1\%$ \\
		\hline
		\hline
\multirow{3}{*}{$(\Z/2\Z)^2$}	&	$5$   & $\frac{13265}{31667}\approx 0.42$  &  $0.49$   & $17\%$   \\
		\cline{2-5}
	&	$7$   & $  \frac{10076}{31667}\approx 0.32$    & $0.37$    &  $14\%$ \\
		\cline{2-5}
	&	$11$   & $\frac{7304}{31667}\approx 0.23$    & $0.25$    &  $7\%$ \\

		\hline
		\hline
\multirow{3}{*}{$(\Z/2\Z)^3$}&	$5$   & $\frac{3931}{5915}\approx 0.67$  &  $0.79$   & $16\%$   \\
		\cline{2-5}
	&	$7$   & $\frac{3191}{5915}\approx 0.54$    & $0.66$    &  $15\%$ \\
		\cline{2-5}
	&	$11$   & $\frac{2417}{5915}\approx 0.41$    & $0.49$    &  $17\%$ \\
		\hline
	\end{tabular}
	\caption{Numerical verification of Conjecture~\ref{conj:independent reg} on the set of fields $K$ such that $\Gal(K)=(\Z/2\Z)^t$, $t=1,2,3$, and conductor $c_K\leq 10^6$ for $t=1$ and $c_K\leq 150000$ for $t=2,3$.}
	\label{tab:reg 2-3}
\end{table}

\subsection{Greenberg's conjecture as a consequence of previous conjectures}\label{ssec:implications}
The Cohen-Lenstra-Martinet heuristic received the attention of many authors and is supported by strong numerical data. Similarly, the Hofmann-Zhang conjecture is backed by the numerical experiments in their paper. In this light, it is interesting to note that these two conjectures imply Greenberg's $p$-rationality conjecture. 

\begin{theorem}\label{th:point 3 of main} Let $t$ be an integer, $q=2$ or $3$ and $p$ a prime such that  $p>4\frac{q^t-1}{q-1}$. Under Conjecture~\ref{conj:independent reg} and Conjecture~\ref{conj:reformulation Martinet}, there exist infinitely many $p$-rational number fields of Galois group $(\Z/q\Z)^t$, or equivalently $\GC_\infty((\Z/2\Z)^t,p)$ and $\GC_\infty((\Z/3\Z)^t,p)$ hold. 
\end{theorem}
\begin{proof}
	Let $\Kappa(D)$ denote the set of totally real number fields of Galois group $(\Z/q\Z)^t$ of conductor less than $D$. Then we have 
\begin{center}
	\begin{align*}
		\limsup_{D\rightarrow \infty} \frac{\#\{K\in \Kappa(D)\text{ non $p$-rational})}{\#\Kappa(D)}&\leq &
		\limsup_{D\rightarrow \infty} \frac{\#\{K\in \Kappa(D)\text{ : }p\mid  h_KR'_{K,p}\}}{\#\Kappa(D)}\\
		&\leq& \Prob(p\mid h_K)+\Prob( p\mid R'_{K,p}). 
		\end{align*}
 		\end{center}
Under Conjecture~\ref{conj:independent reg} we have $$\Prob(p\mid R'_{K,p})=\prod_{k\text{ cyclic subfield}}\Prob(p\mid R'_{k,p}),$$ which is upper bounded by $\frac{q^t-1}{q-1}\mathcal{P}$ where $$\mathcal{P}:=\Prob(p\mid R'_{k,p}\text{ : $k$ is real and }\Gal(k)=\Z/q\Z).$$ Since in each case of the conjecture $\mathcal{P}\leq \frac{2}{p},$ we have  $$\Prob(p\mid R'_{K,p})\leq \frac{q^t-1}{q-1}\cdot \frac{2}{p}.$$  

Under Conjecture~\ref{conj:reformulation Martinet} we have $$\Prob(p\mid h_K)= \prod_{k\text{ cyclic subfield}}\Prob(p\mid h_k),$$ which is upper bounded by $\frac{q^t-1}{q-1}\mathcal{D}$ where $$\mathcal{D}:= \Prob(p\mid h_k\text{ : $k$ is real and }\Gal(k)=\Z/q\Z).$$ In each case of the conjecture we have $\mathcal{D}\leq 1-\prod_{k=2}^\infty(1-\frac{1}{p^k})$ which is upper bounded by $\prod_{k=2}^\infty (\sum_{i=0}^\infty \frac{1}{p^{ki}})-1=\sum_{j=1}^\infty \frac{n(j)}{p^j}$ where $n(j)$ is the number of partitions of $j$ as sums of distinct integers larger than $1$. Since $n(j)\leq 2^j$ we obtain that $\mathcal{D}\leq \sum_{j=2}^\infty(\frac{2}{p})^j\le \frac{8}{p^2}$. Putting all together we obtain
		\begin{align*}			\limsup_{D\rightarrow \infty} \frac{\#\{K\in \Kappa(D)\text{ non $p$-rational})}{\#\Kappa(D)}
		 &\leq & \frac{q^t-1}{q-1}(\frac{2}{p}+\frac{8}{p^2})\\
		 &\leq & \frac{q^t-1}{q-1}\frac{4}{p}<1.
	\end{align*}
\end{proof}

 To conclude this section, we note that Pitoun and Varescon~\cite[Sec. 5]{Pitoun} brought numerical data on the density of $p$-rational quadratic fields. 
 
\section{Algorithmic tools}\label{sec:algorithmic}
Let us make a summary of the algorithms used in the computations of the previous section. The main algorithmic tool in the study of $p$-rational fields is the algorithm of Pitoun and Varescon~\cite{Pitoun} to test $p$-rationality. Their algorithm is not restricted to abelian fields and allows to easily obtain examples of non-abelian $p$-rational fields; in Table~\ref{tab:non-abelian} we list quartic number fields obtained with our implementation of the algorithm~\cite{OnlineComplement}. Since it requires to compute the ray class group, this algorithm is at least as costly as computing the class number. We discuss the complexity of class number algorithms below and conclude that they are computationally expensive. Therefore, in this section we present algorithms which apply to a partial set of number fields but could be much faster in practice. Hence we develop a strategy to decide whether the number fields in a given list are $p$-rational by making as little as possible use of the  complete $p$-rationality test of~Pitoun and Varescon.

 \begin{table}[h!]
		$
		\begin{array}{c|c|c}
		\hline
		\begin{array}{c}
		\text{Galois}\\[-8pt]
		\text{group}
		\end{array}
		& \forall p\leq 100,\text{ }p-\text{rational} &\text{ non $7$-rational}\\
		\hline
		\Z/4\Z     &x^4+x^3+x^2+x+1        &   x^4-23x^3-6x^2+23x+1         \\
		\hline
		V_4          &  x^4-x^2+1            &   x^4+10x^2+1                  \\
		\hline
		D_4          &  x^4-3            &     x^4-6                          \\
		\hline
		A_4          &  x^4+8x+12       &    x^4 - x^3 - 16x^2 - 7x + 27    \\
		\hline
		S_4          &  x^4+x+1          &      x^4+35x+1                     \\
		\hline
		
		\end{array}
		$
		\label{tab:non-abelian}
		\caption{Examples of $p$-rational quartic fiels for each possible Galois group and each prime $p\leq 100$.}
	\end{table}

\subsection{Enumerating all the groups of conductor up to $X$ and $\Gal(K)= (\Z/3\Z)^t$ for some $t$}

Density computations require to list all the number fields up to isomorphism having a given abelian Galois group. Thanks to the conductor-discriminant formula~\cite[Thm 3.11]{Washington}, for any  prime $q$ and integer $t$,  if $\Gal(K)\simeq (\Z/q\Z)^t$ then the conductor of $K$ is $c_K=D_K^\frac{1}{(q-1)q^{t-1}}$. 

For instance, to enumerate the  number fields of Galois group $\Z/3\Z\times \Z/3\Z$ of discriminant less than $X$, we consider the fields $\Q(\zeta_c)$ for each $c\leq X^{1/6}$ such that $c$ is  product of $3$ with exponent $0$ or $2$ and of a set of distinct primes congruent to $1$ modulo $3$ with exponent $1$. For each subgroup $H$ of $(\Z/c\Z)^*$  such that $(\Z/c\Z)^*/H\simeq \Z/3\Z\times\Z/3\Z$, we
compute the fixed field of $H$. 

Cyclic cubic fields of conductor $c$ (and discriminant $c^2$) are obtained by direct formulae in terms of the integer solutions of the equation $u^2+27v^2=4c$ (cf~\cite[Th 6.4.6]{Cohen}). Note that one cannot use the classical parametrization  $P_a(x)=x^3-ax^2-(a+3)x-1$ of the fields of Galois group $\Z/3\Z$ as small conductors can correspond to fields $\Q[x]/P_a(x)$ of large parameters $a$, e.g. the parameters corresponding to the conductors $c_1=6181$ and $c_2=4971871639$ are actually nearly equal:  $a_1=70509$ and $a_2=70510$.    

\subsection{Testing if $p$ divides $h_K$}
 In the context of the Cohen-Lenstra-Martinet heuristic, one has to test if $h_K$ is divisible by~$p$. It is remarkable that for fields of fixed given degree there is no algorithm to compute class numbers faster than computing the value of~$h_K$. Indeed, Buchmann's algorithm~\cite[Algorithm~6.5.9]{Cohen} to compute $h_K$ has an unconditional complexity $O(\sqrt{D_K)}$ and a conjectural complexity $L(D_K)^c$ for a constant $c$, where  $L(X):=\exp(\sqrt{\log X}\sqrt{\log\log X})$.

A second approach due to Fieker and Zhang~\cite{FiekerZhang} tests the divisibility of $h_K$ by $p$ using the $p$-adic class number formula in time $O(D_K^{1/(n_K-1)})$. This is $O(\sqrt{D_K})$ for cyclic cubic fields, which is equal to the proven upper bound on the complexity of Buchmann's algorithm, but slower than the conjectural complexity.   
  
 A third approach is that of Marie-Nicole Gras~\cite{Gras}, which was improved in~\cite{vanderLinden} and~\cite[Eq (5.1)]{Hakkarainen}, and was used to compute the $p$-class group in~\cite{AokiFukuda}.  Based on a result of Hasse, these algorithms compute cyclotomic units (see~\cite[Ch 8]{Washington}) and have a complexity $O(c_K)$ (according to Schwarz's thesis, see~\cite{Hakkarainen}). Due to the conductor-discriminant formula, for cyclic cubic fields this is once again $O(\sqrt{D_K})$. Hence, the cyclotomic unit methods have a complexity which is exponential in $\log D_K$ and therefore larger than the conjectural complexity of Buchmann's algorithm. 
 Obviously, the cost of computations is at least greater than the cost of the binary size of the cyclotomic units, which we compute in the following result.
 \begin{lemma}
 Let $K$ belong to an infinite family of cyclic cubic fields. Let $\sigma\neq \id$ be an automorphism and $u$ a unit such that $\{u,\sigma(u)\}$ generates the group of cyclotomic units $C$ of norm $1$. We identify $u$ with one of its two embeddings in $\R$. Then we have 
 $$\max(|\log |u||,|\log|\sigma(u)||)=D_K^{\frac{1}{4}+o(1)},$$
 where $o(1)$ is a function which tends to zero when $D_K$ tends to infinity. 
 \end{lemma}
 \begin{proof}
Let $\varepsilon$ be a generator of $E$ which is the group of units of $K$  of norm $1$ seen as a $\Z[\zeta_3]$-module (cf. \cite[Sec. 2]{Gras}). Since $\Z[\zeta_3]$ is a P.I.D., there exists $\omega\in \Z[\zeta_3]$ such that $u=\varepsilon^\omega$ and therefore $[E:C]=\Norm_{\Q(\zeta_3)/\Q}(\omega)$ (see Proposition $1$ and the paragraph following it of \cite{Gras}). Here $C$ is the group of cyclotomic units of $K$ of norm $1$. By Hasse's theorem (see~\cite{Gras}), $[E:C]=h_K$, so 
$$\left|\begin{array}{cc} \log |u| & \log |\sigma(u)| \\ \log |\sigma(u)| & \log |\sigma^2(u)| \end{array}    \right|=  \pm R_Kh_K.$$
By the Brauer-Siegel theorem~\cite[Th. 4.9.15]{Cohen} we have $h_KR_K= D_K^{1/2+o(1)}$. The determinant above equals $-(a^2+ab+b^2)$ where $a=\log |u|$ and $b=\log |\sigma(u)|$. Since $\frac{3}{4}\max(|a|,|b|)^2\leq a^2+ab+b^2\leq 3 \max(|a|,|b|)^2$, we obtain the desired result.
 \end{proof}
 Note that in \cite[Sec. 5.8.3]{Cohen}, units are represented in a shorter manner than $|\log |u||,|\log|\sigma(u)||$. However it is not known how to represent cyclotomic units as a product of a number of factors which is polynomial in~$\log D_K$. Also, note that there exist effective lower bounds on the residues at $1$ of the $L$ functions (see for example Louboutin's works), which replace the Brauer-Siegel theorem and imply effective lower bounds on $|\log |u||+|\log \sigma|(u)||$.

 \subsection{Testing the existence of $p$-primary units}
 Schirokauer~\cite{Sch93} proposed a fast method to test if $K$ contains $p$-primary units. In this section $p$ is unramified in $K$. With $K_p$ and $\OO_p$ as in page $3$, let $\lambda:\OO_K\bigcap K_p\rightarrow \OO_p/p\OO_p\simeq \OO_K/p\OO_K$, $x\mapsto \log_p(x)/p\bmod p\OO_p$. Given a basis $(\omega_i)_{1\leq i\leq n_K}$ of an order of $\OO_K$, one can write $\lambda=\sum_i \lambda_i\omega_i$ where $\lambda_i$ are maps into $\F_p$; we call them Schirokauer maps. Note that if $f$ is a monic polynomial and $\alpha$ is a root of $f$ in its number field, then $\Z[\alpha]$ is an order of~$\OO_K$.

\begin{lemma}\label{lemma_reg}
 Let $p$ be an odd unramified prime in the number field $K$. Let $r$ be the unit rank of $K$ and let $U=\{u_1,\ldots,u_r\}$ be a set of units. Assume that $\lambda_1,\ldots,\lambda_n$ are Schirokauer maps corresponding to a basis of the maximal order. We set $M(U):=(\lambda_i(u_j))_{i,j}$.
 
 (1) If $U$ is a system of fundamental units then $K$ has no $p$-primary units if and only if $\rank M(U)=r$. 

(2) If $U$ is an arbitrary set of $r$ units and $\rank M(U)=r$ then $K$ has no $p$-primary units. 
\end{lemma}

\begin{proof}
 (1) Since $p$ is odd and unramified, an element $x\in K$ is a $p$-th power if and only if $\log_p(x)\in p^2\OO_p$. 
 This is equivalent to $\lambda(x)=0$ and also to  $\lambda_1(x)=\ldots=\lambda_{n_K}(x)=0$. The existence of $p$-primary units is hence equivalent to $\ker M(U)\neq 0$ and to $\rank M(U)\neq r$.
 
 (2)  If $\mathcal{E}=(\varepsilon_j)_{j=1,\ldots,r}$ is a system of fundamental units and $\Omega$ is the matrix such that, for each $i$, $u_i=\prod_{j=1}^r \varepsilon_j^{\Omega_{i,j}}$, then $M(U)=\Omega\cdot  M(\mathcal{E})$ so $\rank M(\mathcal{E})\geq \rank M(U)=r$.  \end{proof}

\subsubsection{Fast computation of a unit in cyclic cubic fields}
The remaining question is that of computing a system of generators for $E_K/E_K^p$. In  the  case of cyclic cubic fields the best known method is Buchmann's algorithm~\cite[Alg. 6.5.9]{Cohen}, which has a high cost as discussed in the previous section. We propose a new algorithm to compute units which, although does not work in all the cases, allows us to reduce the total time of the computations when tackling millions of fields. 

\begin{lemma} \label{unit lemma}
	Let $K$ be a number field such that $\Gal(K)\simeq \Z/q\Z$ for an odd prime~$q$. Let $\ell$ be a prime factor of the conductor $c_K$ of $K$ such that $\ell\neq q$. Then the following assertions hold:
\begin{enumerate}
\item there exists an ideal
$\gl$ of $K$
such that $\gl^q=\ell\OO_K$;
\item if $\gl$ is principal, for any generator $\omega\in \OO_K$ of $\gl$ and any generator $\sigma$ of~$\Gal(K/\Q)$, $\frac{\sigma(\omega)}{\omega}$ is a unit.
\end{enumerate}
\end{lemma}
\begin{proof}
(i) Since $\ell$ is
ramified in the Galois field $K$, we have $\ell \OO_K=\gl^e$ for some divisor $e\neq 1$ of $\deg K$. But $\deg K=q$ is prime, so $\ell=\gl^q$.

(ii) The ideal generated by $\frac{\sigma(\omega)}{\omega}$ is $\sigma(\gl)\gl^{-1}$. Since $\sigma\in \Gal(K)$, $\sigma(\gl)$ is a prime ideal above $\ell$. But $\ell$ is totally ramified in $K$, so $\sigma(\gl)=\gl$ and therefore  $\frac{\sigma(\omega)}{\omega}$ is a unit.
\end{proof}

Algorithm~\ref{algo:fast unit} is a direct consequence of this lemma; a sage implementation (\textbf{https://sagemath.org}) is available in the online complement~\cite[algorithm2.sage]{OnlineComplement}.

\begin{algorithm}Fast computation of unit in cyclic cubic number fields.
	\label{algo:fast unit}
	\begin{algorithmic}
		\Require  a cyclic cubic field $K$ and a factorization of its conductor $m$
		\Ensure  a unit of $K$
		\For{$\ell\equiv 1\mod q$ factor of $m$ }  
		\State factor $\ell$ in $\OO_K$ to obtain $\gl$ using \cite[Sec 4.8.2]{Cohen} 
		\State search a generator $\omega_\ell$ of the ideal $\gl$ using LLL~\cite[Alg. 2.6.3]{Cohen}. 
		\EndFor 
		\State \Return a product of the units $\eta_\ell:=\sigma(\omega_\ell)/\omega_\ell$ 
	\end{algorithmic}
\end{algorithm}

We tested Algorithm \ref{algo:fast unit} on  $630$ cyclic cubic number fields listed in Table 1 of~\cite{Gras}, having conductor between $1$ and $4000$. Among them for $272$ fields, (i.e. $43.1 \%$ of $630$ fields), $\gl$ is principal and Algorithm~\ref{algo:fast unit} succeeds. One such example is the field obtained by defining polynomial $x^3+x^2-2x-1$.
Here we write that $\gl$ is principal when there exists a prime factor $\ell$ of the conductor $m$ of the number field $K$ such that $\gl$ is principal.

\section{Conclusion and open questions}
Greenberg's $p$-rationality conjecture for multiquadratic fields and its extension to multicubic fields is suppoted by extensive numerical data and is a consequence of existing conjectures of Cohen-Lenstra-Martinet and Hofmann-Zhang. 

We exhibited an infinite family of cyclic cubic fields without $p$-primary units for a set of primes analoguous to the Wieferich primes. It is an open question to decide if this family has an infinite subset of $p$-rational fields. 

Although we limited our study to abelia fields, one can extend the problem of finding $p$-rational fields to the case of any Galois group. 

Finally, the algorithmic tools for multiquadratic fields which are listed and improved in this work are not restricted to the applications shown in this work. For example, the cyclotomic units computations in multiquadratic fields play an important role in the analysis of the lattice-based cryptography~\cite{Bernstein2017}.

\bibliographystyle{plain}
\bibliography{p_rational_field}

\vspace{2cm}
\end{document}